\documentclass[12pt,a4paper]{article}
\usepackage{texdraw}
\usepackage{latexsym,amsthm,amssymb,amscd,amsmath}
\usepackage{float}
\usepackage[colorlinks]{hyperref}
\usepackage[left=2.3cm,bottom=2.7cm,top=2.7cm,right=2.3cm]{geometry}
\newtheorem{thm}{Theorem}[section]
\newtheorem{prop}[thm]{Proposition}

\newtheorem{Def}[thm]{Definition}
\newtheorem{rem}[thm]{Remark}
\newtheorem{ex}{Example}[section]
\newtheorem{cor}{Corollary}[section]

\newcommand{\be}{\begin{equation}}
\newcommand{\ee}{\end{equation}}
\newcommand{\pa}{\partial}
\newcommand{\pxi}{{\pa \over \pa x^i}}

\title{Weighted Projective Ricci Curvature \\ in Finsler Geometry}

\author{T. Tabatabaeifar,  B. Najafi and A. Tayebi}
\numberwithin{equation}{section}
\begin{document}
\maketitle
\begin{abstract}
In this paper, we introduce the weighted projective Ricci curvature  as an extension of projective Ricci curvature introduced by Z. Shen. We characterize the class of Randers metrics of weighted projective Ricci flat curvature.  We find the  necessary and sufficient condition under which a Kropina metric has weighted projective Ricci flat curvature. Finally, we  show that every projectively flat  metric with isotropic weighted projective Ricci and isotropic S-curvature is a Kropina metric or Randers metric.
\end{abstract}

\section{Introduction}

If two Finsler metrics $F_1$ and $F_2$ on a manifold $M$ are projectively equivalent, then the Douglas tensor ${\bf D}$, the Weyl tensor  ${\bf W}$ or the generalized Douglas-Weyl tensor  ${\bf GDW}$ of $F_1$ is the same as that of $ F_2$. This fact leads to the well-known projective invariant classes of Finsler metrics namely, the class of Douglas metrics, Weyl metrics, and generalized Douglas-Weyl metrics \cite{NT2}\cite{TS}. For other Finslerian  projective invariants, see \cite{AZ}, \cite{NT} and \cite{NT3}.

As Ricci curvature $\textbf{Ric}=\textbf{Ric}(x, y)$ and S-curvature  $\textbf{S}=\textbf{S}(x, y)$ play important roles within the projective geometry of Finsler manifolds, the tensors which contain both of them are  more applicable \cite{TN}\cite{TNR}\cite{TRa}.  In \cite{Sh1}, Shen  introduced projective Ricci curvature by using Ricci and S-curvatures of $F$. It is clear that if two Finsler metrics are pointwise projectively related on a manifold  with a fixed volume form, then their projective Ricci curvatures are equal. In other words, the projective Ricci curvature of Finsler metrics on a manifold $M$ is projective invariant with respect to a fixed volume form on $M$. It is remarkable that, the projective Ricci curvature is actually a kind of weighted Ricci curvatures.

In this paper, we extend projective Ricci curvatures to the weighted projective Ricci curvature as follows. For this purpose, let $(M, F)$ be an $n$-dimensional Finsler manifold with the volume form $dV_F=\sigma(x) dx^1\wedge \cdots\wedge dx^n$. We fix a Finsler metric $F_0$ on $M$ with a fixed volume form $dV_{F_0}=\sigma_0(x) dx^1\wedge \cdots\wedge dx^n$. Then we define the \emph{weighted projective Ricci curvature with respect to $F_0$} by
\begin{equation}\label{PRic}
\textbf{WPRic}_0:=\textbf{Ric} + (n-1)\{\mathbb{S}^2+\mathbb{S}_{|k}y^k\},
\end{equation}
where
\begin{equation}\label{PRicajib}
\mathbb{S}:=\frac{1}{n+1}\big[\textbf{S}+d \ln(\Sigma)\big],
\end{equation}
$\Sigma:={\sigma_0}/{\sigma}$, ${\bf S}$ denotes the S-curvature of $F$ and ``$|$" denotes the horizontal derivation with respect to the Berwald connection of $F$. It is interesting that $\textbf{WPRic}_0$ is invariant under projective transformations. Moreover, if  $dV_{F_0}=\lambda dV_F$ for some constant $\lambda$, then weighted projective Ricci curvature is the same as  projective Ricci curvature $\textbf{PRic}=\textbf{PRic}(x, y)$. A Finsler metric $F$ is called \emph{weighted projective Ricci flat} with respect to $F_0$ if $\textbf{WPRic}_0=0$.  We get the following result on weighted projective Ricci curvature.
\begin{thm}\label{thm2}
Let $(M,F)$ be a complete Finsler manifold and  $F_0$ be a fixed Finsler metric on  $M$. Then the weighted projective Ricci curvature of $F$ with respect to $F_0$  and  Ricci curvature of $F$ satisfies
\be
\textbf{WPRic}_0 \ge \textbf{Ric} \ \ \  or  \ \ \ \textbf{WPRic}_0 \leq \textbf{Ric}
\ee
 if and only if the $S$-curvature of $F$ satisfies $\mathbb{S}=0$. In this case, the S-curvature ${\bf S}$ of $F$ is a exact  1-form.
\end{thm}
\bigskip

The completeness in Theorem \ref{thm2} can not be dropped. Indeed, we can not reduce the completeness condition of $F$ to positive completeness  in Theorem \ref{thm2}. For example, assume  that $F$ is the Funk metric on the standard unit ball in Euclidean space $\mathbb{R}^n$. It is known that $F$ is a Randers metric  $F=\alpha+\beta$ given by
\[
\alpha=\frac{\sqrt{|y|^2-(|x|^2|y|^2-<x,y>^2)}}{1-|x|^2}, \,\, \ \ \ \  \beta=\frac{<x,y>}{1-|x|^2},
\]
where $<,>$ and $|.|$ are Euclidean inner product and  Euclidean norm on $\mathbb{R}^n$, respectively. By a direct calculation, we have $\rho =\ln\sqrt{1-|x|^2}$ and $\rho_0=-\beta$. Since $\beta$ is closed, which is equivalent to $s_{ij}=0$, geodesic coefficients of $F$ are reduced to $G^i=\bar{G}^i+P y^i$, where $P=F^{-1}r_{00}$ and $ \bar{G}^i$ are spray coefficients of the Riemannain metric $\alpha$. To compute the weighted projective Ricci flat of Funk metric with respect to its Riemannian part,  we need to compute $\rho_{0|0}$, where ``$|$'' denotes the horizontal covariant derivative with respect to the Berwald connection of $F$.  We have
\begin{equation}
\rho_{0|0}=-\beta_{|k}y^k =\Big(\frac{\partial \beta}{\partial x^k}-G^i_k \frac{\partial \beta}{\partial y^i}\Big)y^k=\Big(\frac{\partial \beta}{\partial x^k}-\bar{G}^i_k \frac{\partial \beta}{\partial y^i}\Big)y^k -2P\beta =\Big(\frac{\alpha}{\alpha +\beta}\Big)r_{00}. \nonumber
\end{equation}
where $G^i_k={\partial G^i}/{\partial y^k}$ and $\bar{G}^i_k={\partial \bar{G}^i}/{\partial y^k}$.  On the other hand, Funk metric is of constant $S$-curvature ${\bf{S}}={(n+1)}/{2}F$, which is equivalent to $r_{00}=\alpha^2 -\beta^2$ (see \cite{Sh3}, for more details).  Therefore, one can see that Funk metric satisfies 
\[
\textbf{WPRic}_0-\textbf{Ric}=\frac{\Big(\textbf{S}+d \ln(\Sigma)\Big)^2}{(n+1)^2}+\frac{\Big(\textbf{S}+d \ln(\Sigma)\Big)_{|k}y^k}{(n+1)} =\frac{(\beta -\alpha)(3\alpha+\beta)}{4}\leq 0,
\]
which means $\textbf{WPRic}_0 \leq \textbf{Ric}$. However, the $S$-curvature of  Funk metric is not a 1-form.
Thus we can not reduce the completeness condition of $F$ to positive completeness of $F$ in Theorem \ref{thm2}.

\bigskip

By Theorem \ref{thm2}, we get the following.
\begin{cor}
Let $(M, F)$ be a complete Finsler manifold. Then ${\bf PRic}\geq{\bf Ric}$   or \linebreak ${\bf PRic}\leq{\bf Ric}$ if and only if\  \ ${\bf S}=0$.
\end{cor}

\bigskip

In Finsler geometry, $(\alpha,\beta)$-metrics are a rich family of Finsler  metric since they are easy to compute and also have many applications in the real world. They express in the form $F =\alpha \phi({\beta}/{\alpha})$, where $\alpha$ is a Riemannian metric,  $\beta$ is a 1-from and
$\phi$ is a smooth positive function on an open interval. There are two important classes among the $(\alpha,\beta)$-metrics, Randers metric and Kropina metric.
In order to find explicit examples of weighted projective Ricci flat Finsler metrics, we study Randers  metrics. Then,  we have the following.
\begin{thm}\label{thm1}
Let $F=\alpha +\beta$ and $F_0=\alpha$ be a Randers and Riemannian metrics on an $n$-dimensional manifold $M$, respectively, where $\alpha=\sqrt{a_{ij}(x)y^iy^j}$ is a Riemannian metric and $\beta=b_i(x)y^i$ is a 1-form on $M$.  Then $F$ is a weighted projective Ricci flat metric with respect to $F_0$ if and only if for some scalar function $c=c(x)$ on $M$  the following hold
\begin{enumerate}
  \item[$(i)$]  $  \overline{{\bf Ric}}=t^m_{~m}\alpha^2+2t_{00}$,
  \item[$(ii)$]  $s^m_{~0;m}=0$,
  \end{enumerate}
where $\overline{{\bf Ric}}=\overline{{\bf Ric}}(x, y)$ denotes the Ricci curvature of $\alpha$. In this case, $F$ has reversible weighted projective Ricci curvature.
\end{thm}

\bigskip

The Kropina metrics  introduced by  Berwald in connection with a 2-dimensional Finsler manifold  with rectilinear extremal and reconsidered  by Kropina \cite{K}. Beside the Randers metric, the Kropina metric is  the simplest  Finsler metrics which have  many wonderful applications in  electron optics with
a magnetic field, dissipative mechanics, irreversible thermodynamics, etc (see \cite{As} and \cite{I}). It is remarkable  that Kropina
metrics together with Randers metrics are C-reducible metrics \cite{bbb}.  Nevertheless,  Randers metrics are regular Finsler metrics but Kropina metrics are Finsler metrics with singularity. In this paper, we  prove the following.
\begin{thm}\label{thmK}
Suppose that $F={\alpha^2}/{\beta}$ and $F_0=\alpha$ are a Kropina and Riemannian metrics on an $n$-dimensional manifold $M$, respectively, where $\alpha=\sqrt{a_{ij}(x)y^iy^j}$ is a Riemannian metric and $\beta=b_i(x)y^i$ is a 1-form on $M$.  Then $F$ is a weighted projective Ricci flat metric with respect to $F_0$ if and only if $\alpha$ and $\beta$ satisfy the following equations
\begin{eqnarray}
&{\overline{\bf Ric}}=\frac{\lambda}{(n+1)^2b^4} \alpha^2 -\frac{n-2}{b^4}\Big[b^2 s_{0;0}-(s_0+\beta \sigma)^2+b^2\beta \sigma_0\Big]-\frac{n-1}{(n+1)^2}\Big[\theta^2+(n+1)\theta_{|0}\Big],\\
&s^ms_m=-\frac{1}{2}b^2\ t^m_{~m},
\end{eqnarray}
where
\[
\lambda(x):=\frac{n-4}{2}t^m_{~m}-(n-2)\sigma^2+s^m_{~;m}-\sigma_m b^m
\]
and\ \ $\overline{{\bf Ric}}=\overline{{\bf Ric}}(x, y)$ denotes the Ricci curvature of $\alpha$.
\end{thm}

\bigskip

In Theorems \ref{thm1} and \ref{thmK}, we find the necessary and sufficient conditions under which a Randers metric and Kropina metric are  weighted projective Ricci flat. It is interesting to find some condition under which a Finsler metric of isotropic weighted projective Ricci is a Rander or Kropina metric. Let $(M, F)$ be an $n$-dimensional Finsler manifold. Then $F$ is called  of isotropic $\textbf{WPRic}_0$-curvature if
\be
\textbf{WPRic}_0=(n-1)\sigma F^2,
\ee
where $\sigma=\sigma(x)$ is a scalar function on $M$. Here, we find a condition on S-curvature of projectively flat Finsler metric that make it to be a C-reducible metric. More precisely, we prove the following.
\begin{thm}\label{THMPS}
Let $F$ be a projectively flat Finsler metric. Suppose that $F$ has
isotropic weighted projective Ricci and isotropic S-curvature. Then $F$ is a Kropina  or Randers metric.
\end{thm}
\section{Preliminaries}
Given a Finsler manifold $(M, F)$, then a global vector field ${\bf G}$ is associated to $F$ on $TM_0$, which in a standard coordinate $(x^i,y^i)$ for $TM_0$ is given by
\[
{\bf G}=y^i {{\partial} \over {\partial x^i}}-2G^i(x,y){{\partial} \over {\partial y^i}},
\]
where
\[
G^i:=\frac{1}{4}g^{il}\Big[\frac{\partial^2 F^2}{\partial x^k \partial y^l}y^k-\frac{\partial F^2}{\partial x^l}\Big],\ \ \ \  y\in T_xM.
\]
${\bf G}$ is called the  spray associated  to $(M,F)$.  In local coordinates, a curve $c(t)$ is a geodesic of $F$ if and only if its coordinates $(c^i(t))$ satisfy $ \ddot c^i+2G^i(\dot c)=0$. A diffeomorphism $f: (M,F)\rightarrow (M,\bar{F})$ between two Finsler manifolds is called a projective transformation, if $f$ maps every geodesic of $F$ to a geodesic of $\bar{F}$ as a point set.

\bigskip

It is known that a Finsler metric  $F=F(x,y)$ on a convex domain $\mathcal{U}  \subset \mathbb{R}^n$  is  projective  if and only if
its geodesic coefficients $G^i$ are in the form
\[
G^i(x, y) = P(x, y) y^i,
\]
where $P: T{\cal U} = {\cal U}\times \mathbb{R}^n \to \mathbb{R}$ is positively homogeneous with degree one, $P(x,  \lambda y) = \lambda P(x, y)$, $\lambda >0$. We call $P=P(x, y)$ the projective factor of $F(x, y)$.

\bigskip

The notion of Riemann curvature for Riemann metrics can be extended
to Finsler metrics. For a non-zero vector $y \in T_xM_{0}$,
the Riemann curvature $\textbf{R}_y: T_xM \rightarrow T_xM$ is defined by $\textbf{R}_y(u):=R^i_{\ k}(y)u^k {\partial \over {\partial x^i}}$, where
\[
R^i_{\ k}(y)=2{\partial G^i \over {\partial x^k}}-{\partial^2 G^i \over
{{\partial x^j}{\partial y^k}}}y^j+2G^j{\partial^2 G^i \over
{{\partial y^j}{\partial y^k}}}-{\partial G^i \over {\partial
y^j}}{\partial G^j \over {\partial y^k}}.
\]
The family $\textbf{R}:=\{\textbf{R}_y\}_{y\in TM_0}$ is called the Riemann curvature
\cite{Sh3}. We define the Ricci curvature as the trace of $\textbf{R}_y$, i.e., $\textbf{Ric}(x,y):=trace(\textbf{R}_y)$.

\bigskip
For a Finsler metric $F$ on an $n$-dimensional manifold $M$, the
Busemann-Hausdorff volume form $dV_F = \sigma_F(x) dx^1\wedge \cdots\wedge dx^n$ is defined by
\[
\sigma_F(x) := {{\rm Vol} (\Bbb B^n(1))
\over {\rm Vol} \Big \{ (y^i)\in \mathbb{R}^n \ \Big | \ F \Big ( y^i
\pxi|_x \Big ) < 1 \Big \} },
\]
where $\Bbb B^n(1)$ denotes the unit ball in $\mathbb{R}^n$.

\newpage

There is a notion of distortion $\tau =\tau(x,y) $ on $TM$   associated with the Busemann-Hausdorff volume form on $M$, i.e., $dV_{BH}=\sigma(x) dx^1\wedge  dx^2...\wedge dx^n$, which is defined by
\[
\tau  (x,y) = \ln \frac{ \sqrt{ \det (g_{ij}\big(x,y)\big) }}{ \sigma(x) }.
\]
Then the S-curvature is defined by
\[
{\bf S}(x,y) = \frac{d}{dt} \Big [ \tau \big (c(t), \dot{c}(t) \big ) \Big]_{{t=0}},
\]
where $ c=c(t)$ is the geodesic with $c(0)=x$ and $\dot{c}(0)=y$ \cite{Sh1}.
From the definition, we see that the S-curvature ${\bf S}(x,y)$ measures the rate of change of the distortion on $(T_{x}M, F_{x})$ in the direction $y\in T_{x}M$. In a local coordinates,  the S-curvature is given by
\[
{\bf S}= \frac{\pa G^m}{\pa y^m} - y^m \frac{\pa }{\pa x^m} \Big( \ln \sigma\Big).
\]
A Finsler metric $F$ is said to be of  isotropic S-curvature if
\[
{\bf S}=(n+1)cF,
\]
 where $c=c(x)$ is a scalar function on $M$.

\bigskip

The Ricci curvature $\textbf{Ric}=\textbf{Ric}(x, y)$ is the trace of the Riemann curvature defined by $\textbf{Ric}(x, y) := R^m_{\ m}(x, y)$.  A metric $F$ on an $n$-dimensional manifold $M$ is called a weakly Einstein metric if
\be
\textbf{Ric}=(n-1)\Big(\kappa+\frac{3\theta}{F}  \Big) F^2,\label{WE}
\ee
where $\kappa=\kappa(x)$ is a scalar function and $\theta=\theta_i(x)y^i$ is a 1-form on $M$. If $\theta=0$, then $F$ is called an Einstein metric.

Let $(M, F)$ be an $n$-dimensional Finsler manifold. Then the projective Ricci curvature of $F$, denoted it by $\mathbf{{ PRic}}=\mathbf{{ PRic}}(x, y)$ is  defined by following
\begin{align}
\mathbf{{ PRic}}=\mathbf{{Ric}}+\frac{n-1}{n+1}\mathbf{{S}}_{| i}y^i+\frac{n-1}{(n+1)^2}\mathbf{{S}}^2,\label{PRicx}
\end{align}
where ``$|$ '' denotes the horizontal covariant derivative with respect to the Berwarld connection of $F$ (see \cite{CSM}).

\begin{ex}
\emph{Let $\alpha_1=\sqrt{a_{ij}(x)y^iy^j}$ and $\alpha_2=\sqrt{\bar{a}_{ij}(x)y^iy^j}$ be two Ricci-flat Riemannian metrics on the manifolds $M_1$  and $M_2$ of dimension $n_1$ and $n_2$, respectively. Consider the following 4-th root metric
\[
F:=\sqrt[4]{\alpha_1^4 + 2 c \alpha_1^2\alpha_2^2 + \alpha_2^4}.
\]
This is a Ricci-flat $({\mathbf{Ric}}=0)$ and  Berwald metric on the manifold $M:=M_1\times M_2$ of dimension $n:=n_1+n_2$. Thus $F$ is a non-Riemannian Finsler metric satisfies
\[
\textbf{WPRic}_0=(n-1)\{\beta^2+\beta_{|k}y^k\},
\]
where $\beta:=\frac{1}{n+1}d \ln(\Sigma)$ while ${\bf PRic} =0$.}
\end{ex}

\begin{ex}
\emph{Every Ricci-flat Kropina metric is a Berwald metric \cite{ZS}. Berwald metrics have vanishing S-curvature \cite{TR}. Thus
\[
\textbf{WPRic}_0=(n-1)(\theta^2+\theta_{|k}y^k),
\]
where $\theta:=\frac{1}{n+1}\big[d \ln(\Sigma)\big]$. This metric satisfies  ${\bf PRic} =0$.}
\end{ex}

\begin{ex}
\emph{Denote generic tangent vectors on $\mathbb{S}^3$ as
\[
u \frac{\partial}{\partial x}+v \frac{\partial}{\partial y}+  w \frac{\partial}{\partial z} \ .
\]
The Finsler function for Bao-Shen's Randers space is given by
$$F(x,y,z;u,v,w) = \alpha(x,y,z;u,v,w) + \beta(x,y,z;u,v,w) , $$
with
\[
\alpha=  \frac{ \, \sqrt{ \varrho ( c u - z v + y w )^2 + ( z u + c v - x w )^2+ (-y u + x v + c w )^2 } \, }{ 1 + x^2 + y^2 + z^2 } \ ,
\]
\[
\beta=  \frac{ \, \pm \sqrt{ \, \varrho-1 \, } \ ( c u - z v + y w ) \, } { \, 1 + x^2 + y^2 + z^2 \, } ,
\]
where $\varrho>1$ is a real constant \cite{BS}. The family of  Randers metrics on $\mathbb{S}^3$ constructed by Bao-Shen satisfies ${\bf S}=0$. Since these metrics are of constant flag curvature ${\bf K}$, then ${\bf Ric}=2{\bf K}F^2$. Thus Bao-Shen's metrics have constant projective Ricci curvature with $\kappa={\bf K}=constant$. Also, we get
\[
\textbf{WPRic}_0:=2{\bf K}F^2 + 2\{\gamma^2+\gamma_{|k}y^k\},
\]
where $\gamma:={1}/{4}\big[d \ln(\Sigma)\big]$. This metric satisfies  ${\bf PRic} =2{\bf K}F^2$.}
\end{ex}

\begin{ex}
\emph{For a real number $a\in \mathbb{R}^n$, let  us define the Randers metric $F:=\alpha+\beta$ by
\begin{eqnarray*}
\alpha\!\!\!\!\!&:=&\!\!\!\!\! {\sqrt{(1-|a|^2|x|^4)|y|^2+(|x|^2<a,y>-2<a,x><x,y>)^2}\over 1 -|a|^2|x|^4}\\
\beta \!\!\!\!\!&:=&\!\!\!\!\! {|x|^2<a,y>-2<a,x><x,y>\over 1 -|a|^2|x|^4}.
\end{eqnarray*}
Let us put $c:=<a, x>$, $c_0:=c_{x^m}y^m$ and $\rho:=3<a,x>^2-2|a|^2|x|^2$. Then the above  Randers metric satisfies
${\bf S}=(n+1)cF$ and ${\bf Ric}=(n-1)(3c_0F+\rho F^2)$.  See \cite{CS}.  Let us put
\[
\beta:=\frac{1}{n+1}d \ln(\Sigma) \ \ \  \textrm{and }\ \ \ \  \beta_0:=\beta_{|k}y^k.
\]
Then we get
\[
\textbf{WPRic}_0=(n-1)\Big[(\rho +c^2)F^2+(4c_0+2c\beta)F+(\beta_0+\beta^2)\Big].
\]
 This metric satisfies  ${\bf PRic} =(n-1)[4c_0+(\rho+c^2) F]F$.}\end{ex}

\section{Proof of Theorem \ref{thm2}}\label{SEC3}
Let us express the projective factor of two projectively equivalent Finsler metrics in terms of their geometric quantities. Suppose that $F_1$ and $F_2$ are two projectively equivalent Finsler metrics on an $n$-dimensional manifold $M$, with volume forms $dV_{F_1}=\sigma_1(x) dx^1\wedge \cdots\wedge dx^n$ and $dV_{F_2}=\sigma_2(x) dx^1\wedge\cdots\wedge dx^n$, respectively. Suppose that $\textbf{S}_1$ and $\textbf{S}_2$ denote the S-curvatures of $F_1$ and $F_2$, respectively.  Then the projective factor is given by
\[
P=\frac{2}{n+1}\left[(\textbf{S}_1-\textbf{S}_2)-d\ln \left(\frac{\sigma_1}{\sigma_0}\right)+d\ln \left(\frac{\sigma_2}{\sigma_0}\right)\right].
\]
Let $\textbf{G}_1$ and $\textbf{G}_2$ denote the sprays  of $F_1$ and $F_2$, respectively. By assumption,  $\textbf{G}_2=\textbf{G}_1+PY$, where $\textbf{Y}=y^i{\pa}/{\pa y^i}$ is the Liouville vector field. Then by direct calculation we have
\[
\textbf{G}_2+\frac{2}{n+1}\left[\textbf{S}_2-d\ln \left(\frac{\sigma_2}{\sigma_0}\right)\right]\textbf{Y}=\textbf{G}_1+\frac{2}{n+1}\left[\textbf{S}_1-d\ln \left(\frac{\sigma_1}{\sigma_0}\right)\right]\textbf{Y}.
\]
This means that the following spray
\[
\tilde{\textbf{G}}:= \textbf{G}+\frac{2}{n+1}\left[\textbf{S}-d\ln \left(\frac{\sigma}{\sigma_0}\right)\right]\textbf{Y}
\]
is projectively invariant, where ${\bf S}={\bf S}(x, y)$ is the S-curvature of $F$. It is easy to see that $\tilde{\textbf{G}}$ is a globally defined spray on $M$. Let us define
\begin{align}
\widetilde{W}^0_i:=-\frac{1}{2}\Big\{2\widetilde{\textbf{R}}_{x^i}-y^j\widetilde{\textbf{R}}_{x^jy^i}+2\widetilde{G}^k\widetilde{\textbf{R}}_{y^iy^k}
-\widetilde{N}^k_i\widetilde{\textbf{R}}_{y^k}\Big\},
\end{align}
where $\widetilde{\textbf{R}}:=\widetilde{\textbf{Ric}}$ and $\widetilde{N}^i_j:=\partial\widetilde{G}^i/{\partial y^j}$ is the connection coefficients of $\widetilde{\textbf{G}}$. Thus
\[
\widetilde{\textbf{R}}=\textbf{Ric}+\textbf{S}_{|m}y^m+\textbf{S}^2.
\]
Now, we construct a new spray on $M$ as follows
\begin{align}
\widehat{\textbf{G}}&:=\widetilde{\textbf{G}}+\frac{2}{n+1}\beta \textbf{Y},
\end{align}
where $\beta:=d \ln({\sigma}/{\sigma_0})=b_i(x) y^i$ is a closed 1-form on $M$. In this case, we get
\begin{align}
\widehat{\textbf{R}}&:=\widetilde{\textbf{R}}+\frac{1}{3}\beta_{|m}y^m+\frac{1}{9}\Big\{\beta^2+2\beta \textbf{S}\Big\}.
\end{align}
Put
\[
\widehat{H}:=\frac{1}{3}\beta_{|m}y^m+\frac{1}{9}\Big\{\beta^2+2\beta \textbf{S}\Big\}.
\]
Then we have
\begin{align*}
\widehat{W}_i&=\frac{1}{2}\Big\{2\widehat{\textbf{R}}_{x^i}-y^j\widehat{\textbf{R}}_{x^jy^i}+2\widehat{G}^k\widehat{\textbf{R}}_{y^jy^k}-\widehat{N}^k_i\widehat{\textbf{R}}_{y^k}\Big\}\\
&=\widetilde{W}_i-\frac{1}{2}\Big\{2\widehat{H}_{x^i} -y^j\widehat{H}_{x^jy^i}+\frac{2}{3}\beta \widetilde{\textbf{R}}_{y^i}+\widehat{G}^k\widehat{H}_{y^iy^k} -\frac{2}{3}\big(2b_i\widetilde{\textbf{R}}+\beta \widehat{\textbf{R}}_i\big)-\widehat{N}^k_i\widehat{H}_k\Big\},\\
\widehat{N}^k_i&=\widetilde{N}^k_i+\frac{2}{3}\left(b_iy^k+\beta \delta^k_i\right),\\
\widehat{\textbf{R}}_k&=\widetilde{\textbf{R}}_k+\widehat{H}_k.
\end{align*}
\begin{prop}
For 2-dimensional Finsler metric weighted projective Ricci curvature  is dependent on its volume form and $W^0_i$ is independent of volume form.
\end{prop}
\begin{proof}
Let $(M, F)$ be a $2$-dimensional Finsler manifold. Suppose that ${\bf G}={\bf G}(x, y)$ denotes the spray of $F$, and $d\mu=\sigma dx dy$ be its volume form. In a standard local coordinate system $(x, y, u, v)$, one can express ${\bf G}$ as follows
\[
{\bf G}=u\frac{\partial}{\partial x}+v\frac{\partial}{\partial y}-2G(x, y, u, v)\frac{\partial}{\partial u}-2H(x, y, u, v)\frac{\partial}{\partial v}.
\]
We construct a local spray
\[
\widetilde{\textbf{G}}=u\frac{\partial}{\partial x}+v\frac{\partial}{\partial y}-2\tilde{G}(x,y,u,v)\frac{\partial}{\partial u}-2\tilde{H}(x,y,u,v)\frac{\partial}{\partial v},
\]
where
\[
\widetilde{G}:=0, \qquad \widetilde{H}:=-\frac{1}{2}u^2\phi\Big(x,y,\frac{v}{u}\Big).
\]
Since $\widetilde{\textbf{G}}$ is pointwise projective to $\bf{G}$ and $\textbf{W}^0$ is a projective invariant, we just need to calculate the $\textbf{W}^0$ for $\widetilde{\textbf{G}}$. Let $d\mu_0= \sigma_0(x,y)dx dy$ be a fix volume form for $\widetilde{\textbf{G}}$. Then,  we have
\[
\widetilde{\textbf{S}}=-\frac{1}{2}u \phi_\xi -\frac{\sigma_x}{\sigma}u-\frac{\sigma_y}{\sigma}v.
\]
Thus
\begin{align*}
\widehat{\textbf{G}}&=\widetilde{\textbf{G}}+\frac{2}{3}\left[\widehat{\textbf{S}}+d\ln\left(\frac{\sigma}{\sigma_0}\right)\right]\textbf{Y}\\
& =u\frac{\partial}{\partial x}+v\frac{\partial}{\partial y}-2\widehat{G}(x,y,u,v)\frac{\partial}{\partial u}-2\widehat{H}(x,y,u,v)\frac{\partial}{\partial v}
\end{align*}
where
\[
\widehat{G}=\left\{\frac{1}{6}u\phi_\xi+\frac{1}{3}\left(\frac{\sigma_x u+\sigma_y v}{\sigma_0}\right)\right\}u,\qquad\ \ \widehat{H}=-\frac{1}{2}u^2\phi +\left\{\frac{1}{6}u\phi_\xi+\frac{1}{3}\left(\frac{\sigma_x u+\sigma_y v}{\sigma_0}\right)\right\}v,
\]
where $\xi={v}/{u}$. The weighted Ricci of $\widehat{\textbf{G}}$ is given by
\begin{equation}\label{RRR}
\widehat{R}=2(\widehat{G}_x+\widehat{H}_y)-P_xu-P_yv+2P_u\widehat{G} +2P_v\widehat{H}-(\widehat{G}_u\widehat{G}_u+2\widehat{G}_v\widehat{H}_u+\widehat{H}_v\widehat{H}_v)
\end{equation}
where $P:=\hat{G}_u+\hat{H}_v$. By direct calculation, we get
\begin{align}
\nonumber u^{-2}\widehat{R}=&-\phi_y+\frac{1}{3}\phi_{x\xi}+\frac{1}{3}\xi \phi_{y\xi}+\frac{1}{3}\phi \phi_{\xi \xi}-\frac{2}{9} \phi_\xi \phi_\xi -\frac{1}{3}\frac{\sigma_{0_y}}{\sigma_0}\phi\\
&+\frac{1}{9}\left[\frac{\sigma_{0_x}+\sigma_{0_y}\xi}{\sigma_0}\right]\phi_\xi-\frac{\sigma_{0_{xx}}+2\sigma_{0_{xy}}\xi+\sigma_{0_{yy}}\xi^2}{2\sigma_0}
+\frac{4}{9}\left[\frac{\sigma_{0_x}+\sigma_{0_y}\xi}{\sigma_0}\right]^2.\label{RRR1}
\end{align}
It follows that $\widehat{R}$ depend on the volume form $d\mu_0$.

Now, we are going to compute $\widehat{W}^0$. Since $W^0_y(y)=0$, then we have $\widehat{W}^0_1=-\widehat{W}^0_2\xi$. Therefore we only compute $\widehat{W}^0_2$.  The following holds
\begin{equation}\label{WWW}
\widehat{W}^0_2=-\frac{1}{2}\left\{2\widehat{R}_y-u\widehat{R}_{xv}-v\widehat{R}_{yv} +2\widehat{G}\widehat{R}_{uv}+2\widehat{H}\widehat{R}_{vv} -\widehat{G}_v\widehat{R}_u -\widehat{H}_v\widehat{R}_v\right\}.
\end{equation}
Using (\ref{WWW})  we have
\begin{align*}
u^{-2}\widehat{W}^0_2=& \frac{1}{6}\phi_{xx\xi\xi}+\frac{1}{3}\xi\phi_{xy\xi\xi}+\frac{1}{6}\phi_x\phi_{\xi\xi\xi}+\frac{1}{3}\phi \phi_{x \xi\xi\xi}+\frac{1}{6}\xi^2\phi_{yy\xi\xi}+\frac{1}{6}\xi \phi_y\phi_{\xi\xi\xi}\\
&+\frac{1}{3}\xi \phi \phi_{y\xi\xi\xi}+\frac{1}{6}\phi^2\phi_{\xi\xi\xi\xi} -\frac{1}{2}\phi \phi_{y\xi\xi}-\frac{1}{6} \phi_\xi \phi_{x\xi\xi}-\frac{1}{6}\xi\phi_x\phi_{y\xi\xi}\\
&+\frac{2}{3}\phi_\xi\phi_{y\xi}-\frac{2}{3}\xi \phi_{yy\xi}-\frac{2}{3}\phi_{xy\xi}-\frac{1}{2}\phi_y\phi_{\xi\xi}+\phi_{yy}
\end{align*}
which shows that $\widehat{W}^0_i$ is independent of the choice of volume form.
\end{proof}


\bigskip

\noindent
{\bf Proof of Theorem \ref{thm2}:} Let $F$ be a Finsler metric and $\textbf{G}$ be its spray. It is easy to see that
\[
\tilde{\textbf{G}}=\textbf{G}+\frac{2}{n+1}\bigg[\textbf{S}+d \ln(\Sigma)\bigg]\textbf{Y}
\]
 is projectively related to $\textbf{G}$, where $\textbf{Y}=y^i{\pa}/{\pa y^i}$ is the Liouville vector field  and $\Sigma:={\sigma}/{\sigma_0}$. The Ricci curvature of $\tilde{\textbf{G}}$ is called the weighted projective Ricci curvature with respect to $F_0$ (actually with respect to the volume form $dV_0$) and is given by
\begin{equation}\label{PRic}
\textbf{WPRic}_0:=\textbf{Ric} + (n-1)\{\mathbb{S}^2+\mathbb{S}_{|k}y^k\},
\end{equation}
where
\[
\mathbb{S}:=\frac{1}{n+1}\big[\textbf{S}+d \ln(\Sigma)\big],
\]
Now, suppose that $\textbf{WPRic}_0 -\textbf{Ric}\ge 0$. Then by (\ref{PRic}) we obtain
\begin{equation}\label{S1}
\frac{\Big(\textbf{S}+d \ln(\Sigma)\Big)^2}{(n+1)^2}+\frac{\Big(\textbf{S}+d \ln(\Sigma)\Big)_{|k}y^k}{(n+1)}\ge 0
\end{equation}
or
\begin{equation}\label{S1}
\frac{\Big(\textbf{S}+d \ln(\Sigma)\Big)^2}{(n+1)^2}+\frac{\Big(\textbf{S}+d \ln(\Sigma)\Big)_{|k}y^k}{(n+1)}\leq 0.
\end{equation}
Let $c=c(t)$ be the geodesic of $F$ and $\dot{c}(0)=y$ where $y$ is a fixed vector in $T_xM_0$. By assumption, $F$ is complete, hence $c=c(t)$ is defined for all $t\in \mathbb{R} $. We can rewrite (\ref{S1}) as follows
\[
\varphi(t)^2+\varphi'(t)\ge 0,
\]
where
\[
\varphi(t)=\frac{\Big(\textbf{S}+d \ln(\Sigma)\Big)\big(c(t),\dot{c}(t)\big)}{n+1}.
\]
Let
\[
\psi(t):=\frac{\varphi(0)}{1+t\varphi(0)}.
\]
Then it is easy to see that
\[
g^2(t)+g'(t)=0\ \ \ \ \textrm{and} \ \ \  \  g(0)=\varphi(0).
\]
To compare $\varphi(t)$ and $\psi(t)$, we  define
\[
\Phi(t):=\exp\Bigg\{\int_0^t (\varphi(s)+\psi(s))ds\Bigg\}\Big(\varphi(t)-\psi(t)\Big).
\]
We have $\Phi(0)=0$ and
\[
\Phi'(t)=\exp\Bigg\{\int_0^t \big(\varphi(s)+\psi(s)\big)ds\Bigg\} \Big(\varphi(t)^2+\varphi'(t)\Big)\ge 0.
\]
Thus we have $\Phi(t)\ge 0$ for $t\ge 0$ and $\Phi(t)\le 0$ for $t\le0$, which means
\begin{align*}
&\varphi(t)\ge \psi(t),\quad t\ge 0\\
&\varphi(t)\leq \psi(t),\quad t\le 0.
\end{align*}
Suppose that $\varphi(0)\neq 0$ and put $t_0={1}/{\varphi(0)}$. If $\varphi(0)>0$, then $t_0>0$ and we have $\varphi(t_0)\ge \lim_{t\rightarrow t^-_0} g(t)=\infty$ and if $\varphi(0)<0$, then $t_0<0$ and we have $\varphi(t_0)\le \lim_{t\rightarrow t^+_0} \psi(t)=-\infty$ which are impossible. Thus,  we have $\varphi(0)=0$ and consequently $\textbf{S}=-d \ln(\Sigma)$. This completes the proof.
\qed

\bigskip

\begin{cor}
Let $(M,F)$ be a complete Finsler manifold. Suppose $S$-curvature is  reversible, i.e., ${\bf S}(x,-y)={\bf S}(x,y)$, $\forall y\in TM_0$ and  $F_0$ be a fixed Finsler metric on $M$. If the  weighted projective Ricci flat with respect to $F_0$ satisfies $\textbf{WPRic}_0 \ge \textbf{Ric} $, then \ \ ${\bf S}=0$.
\end{cor}
\begin{proof}
By Theorem \ref{thm2}, we have $\textbf{S}=-d \ln(\Sigma)$ which means $S$-curvature is a 1-form on $M$. On the other hand, the only reversible 1-form on $M$ is the zero 1-form.  Thus $\textbf{S}=0$.
\end{proof}

\section{Proof of the Theorem \ref{thm1}}
A Randers metric on a manifold $M$ is a Finsler metric in the following form
\begin{equation} \label{randers}
F=\alpha+\beta,
\end{equation}
where $\alpha=\sqrt{a_{ij}(x)y^iy^j}$ is a Riemannian metric and $\beta=b_i(x)y^i$ a 1-form on $M$. Randers metrics have many applications both in mathematics and physics \cite{Ra}.
 It  is interesting to  characterize weighted projective Ricci flat Randers metrics with respect to their Riemannian parts. Suppose $F=\alpha +\beta$ is a Randers  metric on a manifold $M$. We consider $F_0=\alpha$. In this case, we have
\[
\ln(\Sigma)=(n+1)\ln \rho,
\]
where $\rho:=\sqrt{1-b^2}$ and $b$ is the norm of $\beta$ with respect to $\alpha$, that is, $b^2=a^{ij}b_ib_j$, in which $(a^{ij})$ is the inverse  of the positive-definite matrix $(a_{ij})$ \cite{Sh1}. In \cite{CSM}, Cheng characterized Randers metrics with ${\bf PRic}=0$. A Finsler metric $F$ on a manifold $M$ is said to be generalized Berwald metric if there exists a linear connection on $M$ such that its parallel transformations preserve Minkowski norms on tangent spaces of $M$, which are induced from the Finsler metric $F$. For a Randers metric $F=\alpha+\beta$, it is known that $b$ is a constant function if and only if $F$ is a generalized Berwald metric \cite{TB}, which implies that $\textbf{WPRic}_0=\textbf{PRic}$. Therefore, in the sequel,  we suppose
$F=\alpha +\beta$ is not a generalized Berwald metric.

In this section, we are going to prove the Theorem \ref{thm1}. For this aim, for a Randers metric $F=\alpha+\beta$, let us put
\begin{eqnarray*}
r_{ij}:= {1\over 2}  ( b_{i;j}+b_{j;i} ),\!\!\!\!&&\!\!\!\!\  s_{ij} := {1\over 2} ( b_{i;j} - b_{j;i}),\\
r_j := b^i r_{ij}, \   s_j:=b^i s_{ij},\   r_{i0}: = r_{ij}y^j, \!\!\!\!&&\!\!\!\!\ \  s_{i0}:= s_{ij}y^j,\  r_0:= r_j y^j,\  s_0 := s_j y^j,\\
t_{ij}:=s^k_{\ i} s_{kj}, \ \  t_{i0}:=t_{ij}y^j,\, \ t_{00}:=t_{ij}y^iy^j, \!\!\!\!&&\!\!\!\!\ \ q_{ij}:=r^k_{\ i}s_{kj},\, \ q_{i0}=q_{ij}y^j,\, q_{00}=q_{ij}y^iy^j,
\end{eqnarray*}
where ``$; $ '' stands for horizontal covariant derivative with respect to $\alpha$ and
\[
s^k_{\ i}=a^{kt}s_{ti},\,\,\, r^k_{\ i}=a^{kt}r_{ti}.
\]
Also, let us recall some important facts about Randers metrics \cite{Sh1}. The geodesic coefficients of a Randers metric $F=\alpha+\beta$ are given by
\begin{equation}
G^i= \bar{G}^i +\alpha s^i_{\ 0}+\frac{1}{2F}\Big\{r_{00}-2\alpha s_0\Big\}y^i.
\end{equation}
The Ricci curvature of $F=\alpha+\beta$ is given by
\begin{align}
\nonumber \textbf{Ric} &=\overline{{\bf Ric}}+(2\alpha s^m_{~0;m}-2t_{00}-\alpha^2 t^m_{~m})\\
&+(n-1)\bigg\{\frac{3}{4F^2}(r_{00}-2\alpha s_0)^2+\frac{1}{2F}\Big[4\alpha(q_{00}-\alpha t_0)-(r_{00;0}-2\alpha s_{0;0})\Big]\bigg\},\label{Ric}
\end{align}
where $\overline{{\bf Ric}}$ denotes the Ricci curvature of $\alpha$. The $S$-curvature of $F=\alpha+\beta$ is given by the following:
\begin{equation}\label{Sa}
\textbf{S}=(n+1)\Big[\frac{e_{00}}{2F}-(s_0+\rho_0)\Big],
\end{equation}
where $e_{ij}=r_{ij}+s_ib_j+s_jb_i$ and $e_{00}=e_{ij}y^iy^j$.

For a Randers metric, we get
\begin{equation}
\mathbb{S}=\frac{e_{00}}{2F}-s_0-\rho_0+\rho_0=\frac{r_{00}-2\alpha s_0}{2F}.
\end{equation}
By direct calculation we have
\begin{equation}\label{SH}
\mathbb{S}_{|m}y^m=\mathbb{S}_{;m}y^m-2\alpha s^m_{~0}\mathbb{S}_{y^m}-2(\frac{e_{00}}{2F}-s_0)\mathbb{S}.
\end{equation}
(\ref{Sa}) implies that
\begin{align*}
&\mathbb{S}_{;m}y^m=\frac{1}{2F}\big\{r_{00;0}-2\alpha s_{0;0}\big\}-\frac{r_{00}}{2F^2}\big(r_{00}-2\alpha s_0\big),\\
&2\alpha s^m_{~0}\mathbb{S}_{y^m}=\frac{2\alpha}{F}(q_{00}-\alpha t_0)-\frac{\alpha s_0}{F^2}(r_{00}-2\alpha s_0),\\
&2(\frac{e_{00}}{2F}-s_0)\mathbb{S} =\frac{(r_{00}-2\alpha s_0)^2}{2F^2}.
\end{align*}
Substituting the above equations in (\ref{PRic}), we get  the weighted projective Ricci flat of Randers metric $F=\alpha+\beta$ with respect to $\alpha$ as follows
\begin{equation}
\textbf{WPRic}_0=\overline{{\bf Ric}}+2\alpha s^m_{~0;m}-2t_{00}-\alpha^2 t^m_{~m}.\label{PR1}
\end{equation}
Suppose that $\textbf{WPRic}_0=0$. Then, by (\ref{PR1}) we obtain
\begin{equation}\label{a3}
\alpha Irrat  +Rat=0,
\end{equation}
 where
\begin{align*}
Irrat&:=2 s^m_{~0;m},\\
Rat&:=A_2 \alpha^2 +A_0,
\end{align*}
and
\begin{align*}
A_2&=- t^m_{~m},\\
A_0&=\overline{{\bf Ric}} -2 t_{00}.
\end{align*}
The equation (\ref{a3}) is equivalent to $Irrat=0$ and $Rat=0$. From the former,  we  have
\begin{equation}\label{cc}
\overline{{\bf Ric}}=t^m_{~m}\alpha^2+2t_{00}.
\end{equation}
From $Irrat =0$, we obtain $s^m_{~0;m}=0$, which completes the proof.
\qed

\bigskip

\begin{cor}
Let $F=\alpha+\beta$ be a Douglas Randers metric. Then $F$ is of weighted projective Ricci flat if and only if\ \ $\overline{\bf Ric}=0$.
\end{cor}

\begin{proof}
Let $F=\alpha+\beta$ be a Douglas metric. Then by \cite{BM1},  $\beta$ is closed, i. e. $s_{ij}=0$. By substituting $s_{ij}=t_{ij}=0$ in (\ref{PR1}) we obtain
\[
\textbf{WPRic}_0=\overline{{\bf Ric}},
\]
which concludes that $F$ is of weighted projective Ricci flat if and only if $\overline{\bf Ric}=0$.
\end{proof}

\bigskip

A Finsler metric $F=F(x, y)$ on a manifold $M$ is called of reversible weighted projective Ricci curvature if $\textbf{WPRic}_0(x, -y)=\textbf{WPRic}_0(x, y)$. In this case, for a Randers metric we get the following.
\begin{cor}
A Randers metric is of reversible weighted projective Ricci curvature if and only if $div(d\beta)=0$.
\end{cor}

\begin{rem}
\emph{Let $\Omega$ be a convex domain in $\mathbb{R}^n$ and $F$ a Finsler metric
on $\Omega$. $(\Omega, F)$ is called spherically symmetric if
 $O(n)$ acts  isometrically on $(\Omega, F)$. This
definition is equivalent to say that there exists a positive function $\phi(r,u,v)$ such that $F(x,y)=\phi(|x|,|y|,\langle x,y\rangle)$,  where $|x|=\sqrt{\Sigma^n_{i=1}(x^i)^2}$,
$|y|=\sqrt{\Sigma^n_{i=1}(y^i)^2}$ and $\langle
x,y\rangle=\Sigma^n_{i=1}x^iy^i$. If we fix the Euclidean metric, then the weighted projective Ricci curvature reduced to the projective Ricci curvature. In \cite{ZZ}, Zhu-Zhang characterize projective Ricci flat spherically symmetric Finsler metrics.}
\end{rem}
\section{Proof of the Theorem \ref{thmK}}
In this section, we are going to prove the Theorem \ref{thmK}. Consider a Kropina metric $F ={\alpha^2}/{\beta}$ on an $n$-dimensional manifold $M$. In  \cite{X. Zhang and Y. Shen}, the following are obtained
\[
\sigma_F= \left(\frac{2}{b}\right)^n\sigma_\alpha, \ \ \ \  \sigma_\alpha=\sqrt{det (a_{ij})}.
\]
Therefore, for simplicity we consider  $F_0=\alpha$. In order to prove the Theorem \ref{thmK},  we remark some facts about the Kropina metrics.
\begin{rem}
\emph{In \cite{X. Zhang and Y. Shen},  Zhang-Shen proved that every non-Riemannian Ricci flat Kropina metric is a Berwald metric.}
\end{rem}
\begin{rem}\label{rem1}
\emph{For 1-form $\beta = b_i(x) y^i$ on M, we say that $\beta$ is a conformal form with respect to $\alpha$ if
it satisfies $b_{i;j} + b_{j;i} = \lambda a_{ij}$, where $\lambda= \lambda(x)$ is a function on $M$. For a Kropina metric $F$, the following four conditions are equivalent (see \cite{X}):\begin{description}
        \item[](i) $F$ has isotropic S-curvature ${\bf S} = (n + 1)cF$, where $c = c(x)$ is a function on $M$;
        \item[](ii) $r_{00} = k(x)\alpha^2$, where $k = k(x)$ is a function on $M$;
        \item[](iii) ${\bf S} =0$;
        \item[](iv)  $\beta$ is a conformal form with respect to $\alpha$.
      \end{description}}
\end{rem}

\bigskip

Let $G^{i}$ and $\bar{G}^{i}$ denote the geodesic coefficients of $F$ and $\alpha$, respectively. Then $G^{i}$ and $\bar{G}^{i}$ are related by
\begin{eqnarray} \label{Gi}
G^{i}=\bar{G}^{i} - \frac{F}{2} s^{i}_{\ 0} - \frac{1}{2 b^{2} F} (F s_{0} + r_{00}) (2 y^{i} - F b^{i}).
\end{eqnarray}
Note that, in this case
\[
\sigma_{F}= \left(\frac{2}{b}\right)^{n}\sqrt{det a_{ij}}.
\]
The $S$-curvature of $ F={\alpha^{2}}/{\beta}$ is as follows
\begin{eqnarray}\label{S}
\textbf{S}= \frac{n + 1 }{F b^{2}} \left( F r_{0} - r_{00}\right),
\end{eqnarray}
See Proposition $5.1$ in \cite{X. Zhang and Y. Shen}. In  \cite{X. Zhang and Y. Shen}, the formula for Ricci curvature of a Kropina metric is obtained
\begin{eqnarray}\label{RicK}
\textbf{Ric}= \overline{{\bf Ric}}+\textbf{T},
\end{eqnarray}
where
\begin{eqnarray*}
\textbf{T} \!\!\!\!&:= &\!\!\!\! - \frac{\alpha^{2}}{b^{4}\beta} s_{0}r - \frac{r}{b^{4}}r_{00} + \frac{\alpha^{2}}{b^{2} \beta}b^{m}s_{0;m} + \frac{1}{b^{2}}b^{m}r_{00;m} + \frac{n-2}{b^{2}}s_{0;0} + \frac{n-1}{b^{2}\alpha^{2}}\beta r_{00;0}   \nonumber \\
\!\!\!\!&&\!\!\!\! + \frac{1}{b^{2}} \left( \frac{\alpha^{2}}{\beta} s_{0}+ r_{00}\right)r^{m}_{\ m} - \frac{\alpha^{2}}{\beta}s^{m}_{\ 0;m} -\frac{1}{b^{2}}r_{0;0} - \frac{2(2n -3)}{b^{4}}r_{0}s_{0} - \frac{n-2}{b^{4}}s_{0}^{2} \nonumber \\
\!\!\!\!&&\!\!\!\!  - \frac{4(n-1)}{b^{4}\alpha^{2}}\beta r_{00}r_{0}  + \frac{2(n-1)}{b^{4}\alpha^{2}}\beta r_{00}s_{0} +\frac{3(n-1)}{b^{4}\alpha^{4}}\beta^{2}r_{00}^{2} + \frac{2n}{b^{2}}q_{00} + \frac{1}{b^{4}}r_{0}^{2}  \nonumber \\
\!\!\!\!&&\!\!\!\! - \frac{\alpha^{2}}{b^{2}\beta}q_{0} + \frac{n-1}{b^{2}\beta}\alpha^{2}t_{0} -\frac{\alpha^{4}}{2b^{2}\beta^{2}}s^{m}s_{m} - \frac{\alpha^{2}}{b^{2}\beta}s^{m}r_{0m} - \frac{\alpha^{4}}{4\beta^{2}}t^{m}_{\ m}.
\end{eqnarray*}
Let $F= {\alpha^{2}}/{\beta}$ be a Kropina metric. 
Thus
\begin{eqnarray} \label{0}
\textbf{S}_{\mid m} y^{m} \!\!\!\!&=&\!\!\!\! \textbf{S}_{;m} y^{m} + \left[ F s^{m}_{\ 0} + \frac{1}{b^{2} F} ( F s_{0} + r_{00}) ( 2y^{m} - F b^{m})\right] \frac{\partial \textbf{S}}{\partial y^{m}} \nonumber \\
\!\!\!\!&=&\!\!\!\! \textbf{S}_{;m} y^{m} + F s^{m}_{\ 0} \textbf{S}_{y^{m}} +  \frac{2\textbf{S}}{b^{2} F}( F s_{0} + r_{00}) - \frac{1}{b^{2}} (F s_{0} + r_{00}) b^{m}\textbf{S}_{y^{m}}.
\end{eqnarray}
The following hold
\[
r_{m} s^{m}_{\ 0}= q_{0}, \ \ \ r_{m0} s^{m}_{\ 0}= q_{00}, \ \ \ s^{m}_{\ 0} y_{m}= 0.
\]
From (\ref{S}), we obtain
\begin{eqnarray} \label{1}
&&\textbf{S}_{;m} y^{m}= \frac{(n + 1)}{b^{2}} \left[ r_{0;0} - \frac{r_{00;0}}{F} - \frac{r_{00}^{2}}{F \beta  }  - \frac{2}{Fb^{2}} (r_{0} + s_{0}) (Fr_{0} - r_{00}) \right],
\\
&&F s^{m}_{\ 0} \textbf{S}_{y^{m}}= \frac{(n + 1)}{b^{2}} \left[F q_{0} - 2 q_{00} - \frac{r_{00} s_{0}}{\beta} \right],\label{2}
\\
&&\frac{2\textbf{S}}{b^{2}F} (F s_{0} + r_{00})= \frac{2 (n + 1)}{F^{2} b^{4}} (F s_{0} + r_{00})(F r_{0} - r_{00}),\label{3}
\\
&&\frac{1}{b^{2}} ( F s_{0} + r_{00}) b^{m}\textbf{S}_{y^{m}}= \frac{(n + 1)}{b^{4}} ( F s_{0}+ r_{00})\left[ r - \frac{2}{F^{2}} ( Fr_{0} - r_{00}) - \frac{ r_{00}b^{2}}{\alpha^{2}} \right].\label{4}
\end{eqnarray}
Put
\[
\theta:=d \ln(\Sigma)
\]
which is a $1$-form on $M$. Suppose that
\[
\mathbb{S}:=\frac{1}{n+1}(S+\theta)=\frac{1}{Fb^2}(Fr_0-r_{00})+\frac{\theta}{n+1}.
\]
Then by plugging (\ref{1}), (\ref{2}), (\ref{3}) and (\ref{4}) into (\ref{0}) yields
\begin{eqnarray} \label{S|}
\nonumber(n - 1)\mathbb{S}_{\mid m} y^{m}= \frac{(n - 1)}{b^{2}}\Big[r_{0;0}  - \frac{r_{00;0}}{F} + F q_{0} - 2 q_{00} + \frac{2 }{Fb^{2 } }(Fr_{0} - r_{00})(s_{0} - r_{0})\\
-\frac{4r_{00}^2}{F^2b^2}-\frac{1 }{b^{2}}(Fs_{0} + r_{00}) r\Big]-\frac{(n-1)}{(n+1)}\theta_{|m}y^m.
\end{eqnarray}

\bigskip

\noindent  \textbf{Proof of Theorem \ref{thmK}:} For a Kropina metric, we have
\begin{eqnarray}\label{mnbh}
\textbf{WPRic}_0\!\!\!&=&\!\!\! \textbf{Ric} +(n-1)\mathbb{S}_{\mid m}y^{m} + (n-1)\mathbb{S}^{2}.
\end{eqnarray}
Plugging (\ref{S}), (\ref{RicK}) and (\ref{S|}) into (\ref{mnbh}) yields
\begin{eqnarray}\label{WK1}
\textbf{WPRic}_0\!\!\!&=&\!\!\! \overline{{\bf Ric}} + \frac{n - 2}{b^{4}} \left[b^{2}(r_{0;0}  + s_{0;0}) - ( r_{0} + s_{0})^{2}  \right] +  \frac{n-1}{b^4F}\left[b^{2}F^2t_{0}-4r_0r_{00}\right]\nonumber \\
\!\!\!&&\!\!\! +  \frac{2}{b^{2}} q_{00} -\frac{n F}{b^{4}}s_{0}r - \frac{n}{b^{4}}r r_{00} + ( n - 2) \dfrac{F}{b^{2}} q_{0}  - F s^{m}_{\ 0;m} - \dfrac{F^{2}}{4} t^{m}_{\ m}-\frac{F}{b^2}s^mr_{0m}\nonumber \\
\!\!\!&&\!\!\!     + \frac{F}{b^{2}}b^{m} s_{0;m}  + \frac{1}{b^{2}} b^{m}r_{00;m}  + \frac{1}{b^{2}} (F s_{0} + r_{00}) r^{m}_{\ m} - \frac{F^{2}}{2b^{2}} s^{m}s_{m}-\frac{F}{b^2}s^mr_{0m}\nonumber \\
\!\!\!&&\!\!\! + \dfrac{n-1}{n+1}\Big[ \theta_{|0} +  \dfrac{2 \theta}{F b^{2}} ( F r_{0} - r_{00}) + \dfrac{1}{n+1} \theta^{2} \Big].
\end{eqnarray}
By assumption $\textbf{WPRic}_0=0$. Thus contracting (\ref{WK1}) with $4(n+1)^2b^4F\beta^3$ implies
\begin{eqnarray}\label{WK2}
\!\!\!&&\!\!\! (n+1)^2\Big[4\overline{{\bf Ric}}\ b^4\alpha^2\beta^2 + 4(n - 2) \left[ b^2(r_{0;0}  +  s_{0;0}) - ( r_{0} + s_{0})^{2}  \right]\alpha^2\beta^2 + 4(n-1)[b^2 t_{0}\alpha^4\beta\nonumber \\
\!\!\!&&\!\!\!   -4\beta^3 r_0r_{00}]
+  8b^2 q_{00}\alpha^2\beta^2 -4ns_{0}r \alpha^4\beta - 4nr r_{00}\alpha^2\beta^2 + 4(n - 2)b^2q_{0}\alpha^4 \beta  - 4b^4s^{m}_{\ 0;m}\alpha^4\beta
\nonumber \\
\!\!\!&&\!\!\!  - b^4 t^{m}_{\ m}\alpha^6 + 4b^2b^{m}( s_{0;m}\alpha^4\beta+r_{00;m}\alpha^2\beta^2)  + 4b^2 (\alpha^2 s_{0} + r_{00}\beta) r^{m}_{\ m}\alpha^2\beta
-2b^2 s^{m}s_{m}\alpha^6 \nonumber \\
\!\!\!&&\!\!\! -4b^2\alpha^4\beta s^mr_{0m}\Big]+ 4(n-1)b^2\Big[ (n+1)\big\{b^2\theta_{|0} \alpha^2+  2 \theta (r_{0} \alpha^2 - r_{00}\beta)\big\} + b^4\alpha^2 \theta^{2} \Big]\beta^2=0.\ \ \ \  \ \ \
\end{eqnarray}
It is easy to see that (\ref{WK2}) can be written as follows
\be
\Gamma_6\alpha^6+\Gamma_4\alpha^4+\Gamma_2\alpha^2+\Gamma_0=0,\label{K1}
\ee
where
\begin{align}
\nonumber &\Gamma_6:=-(n+1)^2b^2(b^2 t^{m}_{\ m}+2 s^{m}s_{m}),\\
\nonumber &\Gamma_4:=4(n+1)^2\Big[(n-1)b^2 t_{0}-nrs_{0}+(n - 2)b^2q_{0}-b^4s^{m}_{\ 0;m}+b^2(b^{m} s_{0;m}+s_0r^m_{\ m}-s^mr_{0m})\Big]\beta,\\
\nonumber &\Gamma_2:=4\bigg\{(n+1)^2\Big[\overline{{\bf Ric}}\ b^4+(n - 2) \left[ b^2(r_{0;0}  +  s_{0;0}) - ( r_{0} + s_{0})^{2}  \right] +2b^2 q_{00}-nr r_{00}\\
\nonumber &\quad\quad\quad\quad + b^2(b^{m}r_{00;m}+r^m_{\ m}r_{00})\Big]+ (n-1)b^2\Big[ (n+1)b^2\theta_{|0}+  2(n+1) \theta r_{0}   + b^4 \theta^{2} \Big]\bigg\}\beta^2,\\
&\Gamma_0:= -8(n^2-1)\beta^3   r_{00}[ 2(n+1)r_0 +b^2\theta].\label{K22}
\end{align}
Then, (\ref{K1}) can be written as $(\Gamma_6\alpha^4+\Gamma_4\alpha^2+\Gamma_2)\alpha^2=-\Gamma_0$ which means that $\Gamma_0$ should be a multiple of $\alpha$. It conclude that $r_{00}=\sigma \alpha^2$. Therefore by Remark \ref{rem1}, we have ${\bf S}=0$ and by direct calculation we have:
\begin{eqnarray}
\nonumber &q_{00}=0, \quad q_0=\sigma s_0,\quad s^mr_{0m}=\sigma s_0,\quad r_0=\sigma \beta, \\
&r_{00;0}=\sigma_0 \alpha^2,\quad r^m_{~m}=n\sigma \quad r=\sigma b^2,\quad r_{0;0}=\sigma_0\beta+\sigma^2\alpha^2.  \label{KK}
\end{eqnarray}
By substituting (\ref{KK}) in (\ref{K22}), we obtain:
\begin{eqnarray}
\nonumber &&\Gamma'_4:=-(n+1)^2b^2(b^2 t^{m}_{\ m}+2 s^{m}s_{m}),\\
\nonumber &&\Gamma'_2:=4(n+1)^2b^2\Big[(n-1) t_{0}+(n-3)\sigma s_0 -b^2s^{m}_{\ 0;m}+b^{m} s_{0;m} +[(n-2)\sigma^2 +b^m \sigma_m ]\beta  \big]\beta,\\
\nonumber &&\Gamma'_0:=4\bigg[(n+1)^2\Big[\overline{{\bf Ric}}\ b^4+(n - 2) \left[ b^2(\sigma_0 \beta +  s_{0;0}) - ( \sigma \beta + s_{0})^{2}  \right] \Big]\\
&&\quad \quad \quad \quad + (n-1)b^4 \Big[ (n+1)\theta_{|0}  +  \theta^{2} \Big]\bigg]\beta^2.\label{K11}
\end{eqnarray}
In this case (\ref{K1}) reduces to
\be
\Gamma'_4 \alpha^4+\Gamma'_2 \alpha^2 +\Gamma'_0 =0. \label{KK3}
\ee
Then $\Gamma'_0$ should be a multiple of $\alpha^2$. Since $\beta^2$ cannot be a multiple of $\alpha^2$, then we have
\begin{align}
\nonumber (n+1)^2\Big[\overline{{\bf Ric}}\ b^4&+(n - 2) \left[ b^2(\sigma_0 \beta +  s_{0;0}) - ( \sigma \beta + s_{0})^{2}  \right] \Big]\\
&\ \ \ + (n-1)b^4 \Big[ (n+1)\theta_{|0}  +  \theta^{2} \Big]=\lambda(x)\alpha^2. \label{KK4}
\end{align}
In this case (\ref{K1}) reduces to
\[
(b^2 t^m_{~m}+2s^ms_m)\alpha^2 + (n-1)t_0+(n-3)\sigma s_0+(n-2)\beta \sigma^2 -b^2 s^m_{~0;m}+(\sigma_m \beta+s_{0;m})b^m+\lambda\beta =0.
\]
which means that
\begin{align}
&b^2 t^m_{~m}+2s^ms_m=0,\label{K5} \\
&(n-1)t_0+(n-3)\sigma s_0+(n-2)\beta \sigma^2 -b^2 s^m_{~0;m}+(\sigma_m \beta+s_{0;m})b^m+\lambda\beta=0. \label{K6}
\end{align}
By (\ref{K5}) we obtain
\be
s^ms_m=-\frac{1}{2}b^2\ t^m_{~m}. \label{K10}
\ee
Differentiating both sides of (\ref{K6}) with respect to $y^i$ yields
\be
(n-1)t_i+(n-3)\sigma s_i+(n-2)b_i \sigma^2 -b^2 s^m_{~i;m}+(\sigma_m b_i+s_{i;m})b^m+\lambda b_i=0. \label{K7}
\ee
Contracting (\ref{K7}) with $b^i$ gives
\be
(n-4)b^2t^m_{~m}-2(n-2)b^2\sigma^2+2b^2s^m_{~;m}-2\sigma_m b^mb^2-2\lambda b^2=0. \label{K8}
\ee
By substituting (\ref{K10}) in (\ref{K8}) we have
\be
\lambda=\frac{n-4}{2}\ t^m_{~m}-(n-2)\sigma^2+s^m_{~;m}-\sigma_m b^m \label{K9}
\ee
which completes the proof of Theorem \ref{thmK}.
\qed

\section{Proof of Theorem \ref{THMPS}}
In this section, we are going to prove the Theorem \ref{THMPS}.

\bigskip
\noindent
{\bf Proof of Theorem \ref{THMPS}:}  The Riemannian curvature is defined by
\begin{eqnarray}
R^i_{\ k}(y)=2{\partial G^i \over {\partial x^k}}-{\partial^2 G^i \over
{{\partial x^j}{\partial y^k}}}y^j+2G^j{\partial^2 G^i \over
{{\partial y^j}{\partial y^k}}}-{\partial G^i \over {\partial
y^j}}{\partial G^j \over {\partial y^k}}.\label{Riemann}
\end{eqnarray}
By assumption, there is a local coordinate system $(x^i)$ and a scalar function $P= P(x, y)$ such that $G^i = Py^i$. Putting it in (\ref{Riemann}) implies that
\begin{eqnarray}
\nonumber R^i_{\ k}&=& 2{\partial (Py^i) \over {\partial x^k}}-{\partial^2 (Py^i) \over
{{\partial x^j}{\partial y^k}}}y^j+2Py^j{\partial^2 (Py^i) \over
{{\partial y^j}{\partial y^k}}}-{\partial (Py^i) \over {\partial
y^j}}{\partial (Py^j) \over {\partial y^k}}\\
&=& \Big[2{\partial P \over {\partial x^k}}-{\partial^2 P \over{{\partial x^j}{\partial y^k}}}y^j-P{\partial P \over {\partial y^k}}\Big]y^i
+\Big[P^2 -{\partial P \over {\partial x^j}}y^j\Big]\delta^i_k.\label{Ri}
\end{eqnarray}
Taking a trace of (\ref{Ri}) yields
\be
\textbf{Ric} =(n-1)(P^2-P_0),\label{R1}
\ee
where $P_0:=P_{x^j}y^j$.  For simplicity, let us ut
\[
\eta=\eta_i(x)y^i:=\frac{1}{n+1}d \ln(\Sigma).
\]
By assumption,  $F$ has  isotropic projective weighted Ricci curvature $\textbf{WPRic}_0=(n-1)\sigma F^2$ and isotropic ${S}$-curvature ${\bf S}=(n+1)cF$, where $\sigma=\sigma(x)$ and $c=c(x)$ are scalar functions on $M$. Then by considering (\ref{R1}), the relation (\ref{PRic}) reduces to following
\be
\sigma F^2=(P^2-P_0)+\big[(cF+\eta)^2+(cF+\eta)_{|s}y^s\big].\label{R2}
\ee
Put $c_0:=c_{|s}y^s$ and $\eta_0:=\eta_{|s}y^s$. Then  (\ref{R2}) is rewritten as follows
\be
(\sigma-c^2)F^2-(2c\eta-c_0)F + (P_0-P^2-\eta^2-\eta_0)=0.\label{T1}
\ee
According to (\ref{T1}), we have two main cases as follows:\\\\
{\bf Case (i)}: If $\sigma=c^2$, then (\ref{T1}) implies
\[
F=\frac{P_0-P^2-\eta^2-\eta_0}{2c\eta-c_0}
\]
which is a Kropina metric and is not a regular metric.\\\\
{\bf Case (ii)}: If $\sigma\neq c^2$, then we get
\[
F=\frac{\sqrt{(2c\eta-c_0)^2-4(\sigma-c^2)(P_0-P^2-\eta^2-\eta_0)}+2c\eta-c_0}{2(\sigma-c^2)}
\]
which is a Randers metric. This completes the proof.
\qed

\bigskip
\begin{ex}
\emph{Let $F$ be an Einsteinian Riemannian metric, and $F_0$ be another Riemannian metric on a manifold $M$. Suppose $F$ and $F_0$ have the same volume form. In this case, $F$ is of isotropic projective weighted Ricci curvature.}
\end{ex}

\bigskip

\noindent
 Tayebeh Tabatabaeifar and Behzad Najafi \\
Department of Mathematics and Computer Science\\
Amirkabir University of Technology (Tehran Polytechnic)\\
Tehran. Iran\\
Email:\  t.tabaee@gmail.com\\
Email:\ behzad.najafi@aut.ac.ir\\

\bigskip
\noindent
Akbar Tayebi\\
Faculty  of Science, Department of Mathematics\\
University of Qom\\
Qom. Iran\\
Email:\ akbar.tayebi@gmail.com
\end{document}